\documentclass[reqno,12pt]{amsart}

\usepackage{epsf}
\usepackage{graphics}
\usepackage{graphicx}
\usepackage{amssymb}
\usepackage{amsmath}

\usepackage{tikz, graphics}
\usetikzlibrary{decorations.markings}
\usetikzlibrary{decorations.pathreplacing}

\usetikzlibrary{knots}
\usetikzlibrary{decorations.markings}

\date{}

\theoremstyle{plain}
\newtheorem{theorem}{Theorem}
\newtheorem{corollary}{Corollary}
\newtheorem{proposition}{Proposition}
\newtheorem{lemma}{Lemma}
\newtheorem{rem}{Remark}

\theoremstyle{definition}

\theoremstyle{remark}

 \newcommand{\IN}[0]{\mathbb{N}}

 \newcommand{\CL}[0]{\mathcal{L}}

\title{Arborescence of positive Thompson links}

\author{Valeriano Aiello}
\address{Valeriano Aiello,
Mathematisches Institut, Universit\"at Bern,  Alpeneggstrasse 22, 3012 Bern, Switzerland
}\email{valerianoaiello@gmail.com}
\author{Sebastian Baader} 
\address{Sebastian Baader,
Mathematisches Institut, Universit\"at Bern, Sidlerstrasse 5, 3012 Bern, Switzerland
}\email{sebastian.baader@math.unibe.ch}

\begin{document}

\begin{abstract} 
We show that the links associated with positive elements of the Thompson group $F$ coincide with the closures of bipartite arborescent tangles.
\end{abstract}

\maketitle
\section{Introduction}

The Thompson group  $F$  shares two important features with the union of the braid groups: it contains a natural positive monoid $F_+$ generated by countably many generators $x_0,x_1,x_2,\ldots$, and it comes with a link construction, recently described by Jones~\cite{J1,J2}. The elements of the Thompson group can be encoded by pairs of finite rooted binary plane trees. As in the braid groups, every group element of the Thompson group is a product of a positive and a negative element, each of which is determined by a single tree~\cite{CFP}. We refer to links associated with positive elements of the Thompson group as positive Thompson links. 
In \cite[Problem 6.16]{GS} Golan and Sapir asked what types of links correspond to $F_+$.
As we will see, these links are all arborescent. Moreover, their underlying rooted plane trees are bipartite, in the following sense: their vertices carry weights $\pm 1$, so that all pairs of neighbouring vertices carry different signs, and so that all vertices with positive sign have valency two.

\begin{theorem} \label{arborescent}
The set of positive Thompson links coincides with the set of closures of bipartite arborescent tangles.
\end{theorem}

Arborescent links are described by rooted plane trees with integer weights~\cite{BS}; they are traditionally also called algebraic \cite{T}.
 Our convention is such that alternating tangles 
 correspond to trees all of whose weights
 have the same sign. This is compatible with the usual convention for rational tangles~\cite{Co,GK}. 
Our second result shows that all closures of positive arborescent tangles are realised as positive Thompson links.

\begin{corollary} \label{positive}
The set of positive Thompson links contains the set of arborescent links associated with weighted plane rooted trees with all strictly positive weights.
\end{corollary}

The class of links associated with positive arborescent tangles contains all two-bridge knots, in particular positive and negative ones. In order to prevent confusion, we should point out that  positive/negative vertex signs do not necessarily stand for positive/negative crossings. The two examples shown at the top of Figure~\ref{fig1} illustrate the sign convention and demonstrate that the positive and negative trefoil knots are both realised as arborescent tangles with negative weights.
\begin{figure}[htb]
\begin{center}
\raisebox{-0mm}{\includegraphics[scale=0.8]{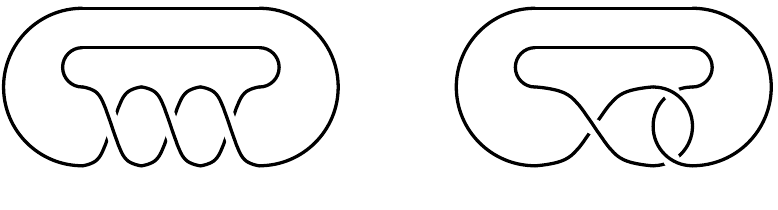}}
\[
\begin{tikzpicture}[x=.5cm, y=.5cm,
    every edge/.style={
        draw,
      postaction={decorate,
                    decoration={markings}
                   }
        }
]

%
%
%
%
%
%
%
%
%
%
%
 
 \node at (1.5,-2) {$\scalebox{.5}{$\,$}$};
 \node at (1.5,-1.5) {$\scalebox{.75}{$\star -3$}$};

\end{tikzpicture}
\qquad
\qquad\qquad
\begin{tikzpicture}[x=.5cm, y=.5cm,
    every edge/.style={
        draw,
      postaction={decorate,
                    decoration={markings}
                   }
        }
]


%
%
%
%
%
%
%
%
%
 
 \node at (1.5,-1.5) {$\scalebox{.5}{$\star$}$};
 \node at (1.4,-2) {$\scalebox{.75}{$-1$}$};
 \node at (2.75,-2) {$\scalebox{.75}{$-2$}$};
  \draw[thick] (1.7,-1.5)--(3,-1.5);

\end{tikzpicture}
\]

\bigskip
\bigskip
\raisebox{-0mm}{\includegraphics[scale=1.2]{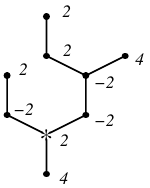}}
\quad
\raisebox{-7mm}{\includegraphics[scale=0.5]{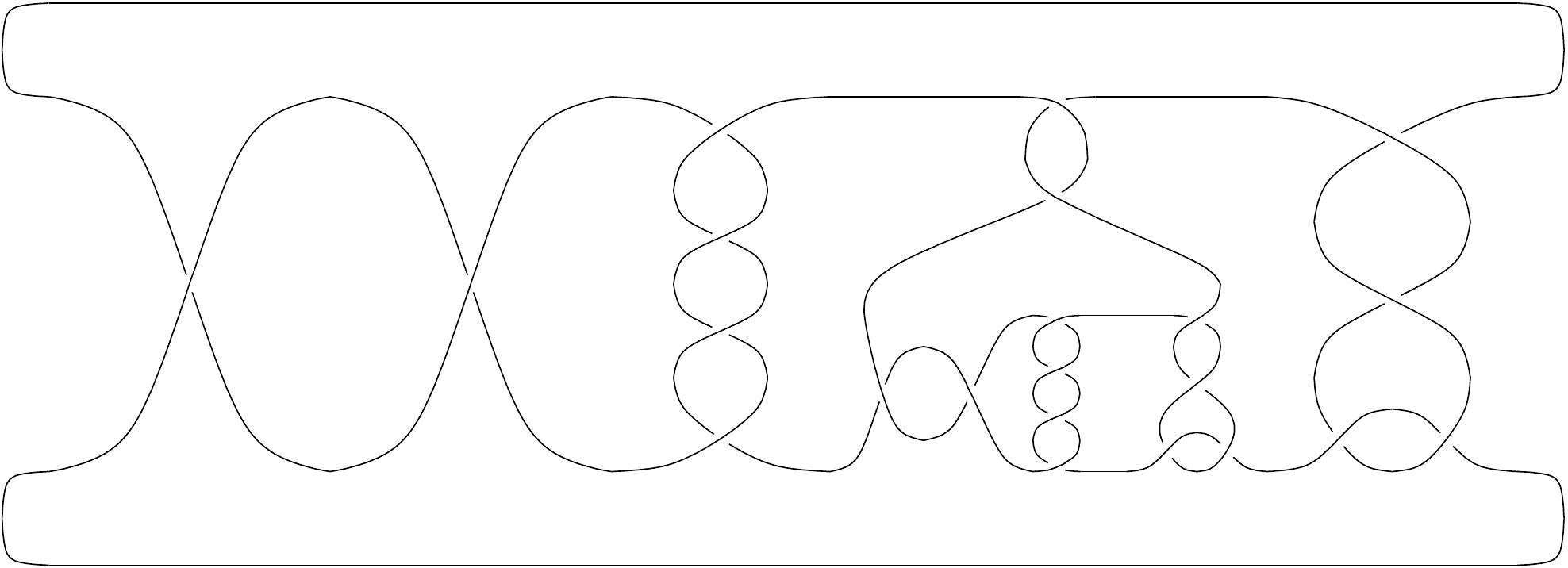}}

\caption{Arborescent knots}\label{fig1}
\end{center}
\end{figure}
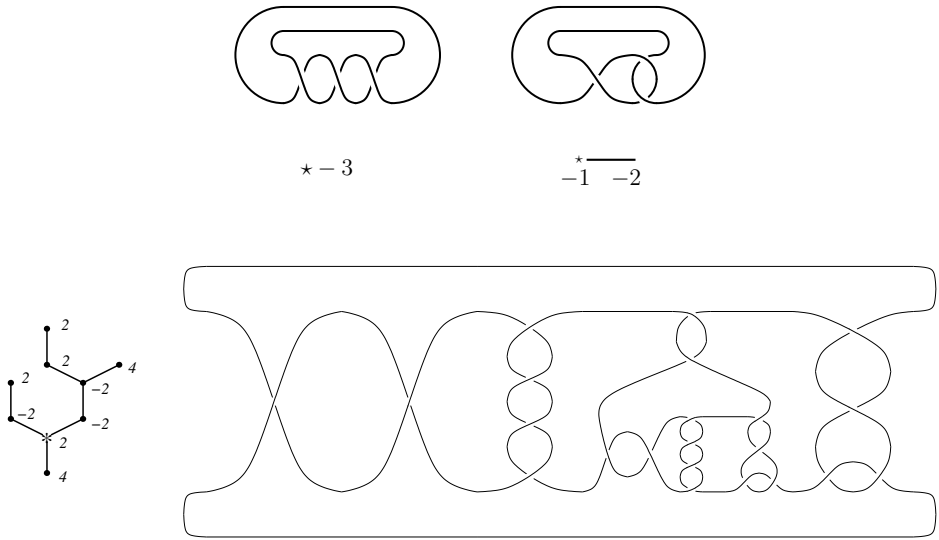

While unoriented links are encoded in the Thompson group $F$, for oriented links one has to consider Jones' oriented subgroup $\vec{F}$ (\cite{J1}, see also \cite{V, ACJ}).
In a predecessor of this paper, we 
proved that links associated with positive elements of the oriented Thompson group are positive~\cite{AB}. The links there are a subfamily of the ones considered here. The examples of Figure~1 show that link positivity does not persist when dropping the orientability.

The proofs of Theorem~\ref{arborescent} and Corollary \ref{positive} are presented in Sections 2 and 3, respectively. 

\section{Arborescent diagrams and bipartite trees}

The Thompson group $F$ consists of all piecewise affine homeomorphisms of the unit interval with slopes being powers of $2$ and dyadic breakpoints. The action of such a homeomorphism can be stored by pairs of finite rooted planar 
binary trees with equal numbers of leaves
~\cite{CFP}. 
A tree with $2$ leaves is herein referred to as caret (see Figure \ref{fig3}). 
We usually denote   such pairs of trees by the symbol $\frac{T_+}{T_-}$ and  as customary, we draw a pair of trees in the plane with the tree $T_+$ upside down on top of the other $T_-$. 
The tree $T_+$ is called the top tree, while $T_-$ is the bottom tree.
Whenever two pairs of trees differ by a sequence of additions/deletions of pairs of opposing carets (see Figure \ref{oppcar}) they are said to be equivalent.
Thanks to this equivalence relation, the rule $\frac{T_+}{T}\cdot \frac{T}{T_-}=\frac{T_+}{T_-}$
defines the multiplication in $F$. 
The trivial element is represented by any pair $\frac{T}{T}$ and the inverse of $\frac{T_+}{T_-}$
is just $\frac{T_-}{T_+}$. 
Each finite rooted planar binary   tree may be obtained from the tree consisting of one vertex by repeated additions of carets. 
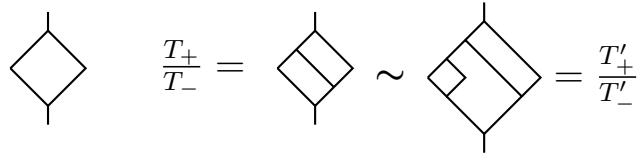
\begin{figure}[htb]
\begin{center}
\[
\begin{tikzpicture}[x=.5cm, y=.5cm,
    every edge/.style={
        draw,
      postaction={decorate,
                    decoration={markings}
                   }
        }
]

 \draw[thick] (1,0)--(2,1)--(3,0)--(2,-1)--(1,0); 
\draw[thick] (2,1)--(2,1.5);
\draw[thick] (2,-1)--(2,-1.5);
\node (bbb) at (-1,-2) {$\scalebox{1.25}{$\,$}$}; 

\end{tikzpicture}
\qquad 
\begin{tikzpicture}[x=.5cm, y=.5cm,
    every edge/.style={
        draw,
      postaction={decorate,
                    decoration={markings}
                   }
        }
]
\node (bbb) at (-1,0) {$\scalebox{1.25}{$\frac{T_+}{T_-}=$}$}; 

 \draw[thick] (1,0)--(2,1)--(3,0)--(2,-1)--(1,0); 
 \draw[thick] (1.5,0.5)--(2.5,-.5);
\draw[thick] (2,1)--(2,1.5);
\draw[thick] (2,-1)--(2,-1.5);

\node (bbb) at (-1,-2) {$\scalebox{1.25}{$\,$}$}; 

\end{tikzpicture}\, \,
 \begin{tikzpicture}[x=.5cm, y=.5cm,
    every edge/.style={
        draw,
      postaction={decorate,
                    decoration={markings}
                   }
        }
]
\node (bbb) at (4.5,0) {$\scalebox{1.25}{$=\frac{T_+'}{T_-'}$}$}; 
\node (bbb) at (-1,0) {$\scalebox{1.25}{$\thicksim$}$}; 

\draw[thick] (0,0)--(.5,.5)--(1,0)--(.5,-.5)--(0,0);
\draw[thick] (.5,0.5)--(1,1)--(2.5,-.5)--(1.5,-1.5)--(0.5,-.5);
\draw[thick] (1,1)--(1.5,1.5)--(3,0)--(2.5,-.5);
\draw[thick] (1.5,-1.5)--(1.5,-2);
\draw[thick] (1.5,1.5)--(1.5,2);

\end{tikzpicture}\]
\caption{A pair of opposing carets and two equivalent pairs of trees.}
\end{center}\label{oppcar}
\end{figure}

A few years ago Jones introduced a machinery that 
allows one to construct
unitary representations of the Thompson group \cite{J1, J16, AJ, ABC}.
This also led him  to define a link diagram out of 
such a pair of trees by gluing together a pair of tangles
 along a horizontal line.
Every link arises from his construction~\cite{J1}.  

We now review Jones's construction  with a simple example. 
Let~$\frac{T_+}{T_-}$ be a pair of rooted binary plane trees with $n$~leaves, for instance
\[
\begin{tikzpicture}[x=.5cm, y=.5cm,
    every edge/.style={
        draw,
      postaction={decorate,
                    decoration={markings}
                   }
        }
]
\node (bbb) at (-1,0) {$\scalebox{1.25}{$\frac{T_+}{T_-}=$}$}; 

 \draw[thick] (2,0)--(2.5,.5)--(3,0); 
\draw[thick] (2.5,0.5)--(3,1)--(4,0);
\draw[thick] (4,0)--(2.5,-1.5)--(1,0)--(2.5,1.5)--(3,1);
\draw[thick] (3,0)--(3.5,-.5);
\draw[thick] (2,0)--(3,-1);
\draw[thick] (2.5,1.5)--(2.5,1.75);
\draw[thick] (2.5,-1.5)--(2.5,-1.75);

\end{tikzpicture}
\]
We are going to associate a tangle to $T_+$. 
For this purpose 
we think of $T_+$ as sitting
in the upper-half plane, with leaves on the positive integers of the $x$-axis. 
First 
we turn the trivalent vertices into $4$-valent ones, then we turn the new $4$-valent vertices into crossings and extend the strand sprouting from the root until it meets the zero point on the $x$-axis. 
\[
\begin{tikzpicture}[x=.5cm, y=.5cm,
    every edge/.style={
        draw,
      postaction={decorate,
                    decoration={markings}
                   }
        }
]
\node (bbb) at (-1,0) {$\scalebox{1.25}{$T_+=$}$}; 

 \draw[thick] (2,0)--(2.5,.5)--(3,0); 
\draw[thick] (2.5,0.5)--(3,1)--(4,0);
\draw[thick] (1,0)--(2.5,1.5)--(3,1);
\draw[thick] (2.5,1.5)--(2.5,1.75);

\end{tikzpicture}\quad
\begin{tikzpicture}[x=.5cm, y=.5cm,
    every edge/.style={
        draw,
      postaction={decorate,
                    decoration={markings}
                   }
        }
]

 \draw[thick] (2,0)--(2.5,.5)--(3,0); 
 \draw[thick] (2.5,.5)--(2.5,0); 
 
\draw[thick] (2.5,0.5)--(3,1)--(4,0);
 \draw[thick] (3,1)--(3.5,0); 

\draw[thick] (1,0)--(2.5,1.5)--(3,1);
\draw[thick] (2.5,1.5)--(2.5,1.75);
\draw[thick] (2.5,1.5)--(1.5,0);

\node (bbb) at (-1,0) {$\scalebox{1.25}{$\mapsto$}$}; 

\end{tikzpicture}
\quad
 \begin{tikzpicture}[x=.5cm, y=.5cm,
    every edge/.style={
        draw,
      postaction={decorate,
                    decoration={markings}
                   }
        }
]
\node (bbb) at (-2,0) {$\scalebox{1.25}{$\mapsto$}$}; 
\node (bbb) at (-3,-.5) {$\scalebox{1.25}{$\,$}$}; 
\draw[thick] (2,0) to[out=90,in=90] (3,0);  
\draw[thick] (2.5,0.5) to[out=90,in=90] (4,0);  
\draw[thick] (1,0) to[out=90,in=90] (3.5,.75);  
\draw[thick] (0,0) to[out=90,in=90] (1.5,.95);

\draw[thick] (2.5,0)--(2.5,.2); 
\draw[thick] (3.5,0)--(3.5,.5); 
\draw[thick] (1.5,0)--(1.5,.6);

\end{tikzpicture}
\]
The tangle associated with $T_+$ is a union of $n$~half-circles in the upper half-plane with endpoints on the $x$-axis and with the innermost circles passing on top
(thus the tangle is alternating). 

 Repeat the same for $T_-$ in the lower half-plane. 
 \[
\begin{tikzpicture}[x=.5cm, y=.5cm,
    every edge/.style={
        draw,
      postaction={decorate,
                    decoration={markings}
                   }
        }
]
\node (bbb) at (-1,0) {$\scalebox{1.25}{$T_-=$}$}; 

\draw[thick] (4,0)--(2.5,-1.5)--(1,0);
\draw[thick] (3,0)--(3.5,-.5);
\draw[thick] (2,0)--(3,-1);
\draw[thick] (2.5,-1.5)--(2.5,-1.75);

\end{tikzpicture}\quad
\begin{tikzpicture}[x=.5cm, y=.5cm,
    every edge/.style={
        draw,
      postaction={decorate,
                    decoration={markings}
                   }
        }
]

\draw[thick] (4,0)--(2.5,-1.5)--(1,0);
\draw[thick] (3,0)--(3.5,-.5);
\draw[thick] (2,0)--(3,-1);

\draw[thick] (2.5,-1.5)--(2.5,-1.75);

\draw[thick] (2.5,-1.5)--(1.5,0);
\draw[thick] (3,-1)--(2.5,0);
 \draw[thick] (3.5,-.5)--(3.5,0);

\node (bbb) at (-1,0) {$\scalebox{1.25}{$\mapsto$}$}; 

\end{tikzpicture}\quad
 \begin{tikzpicture}[x=.5cm, y=.5cm,
    every edge/.style={
        draw,
      postaction={decorate,
                    decoration={markings}
                   }
        }
]
\node (bbb) at (-2,-0.25) {$\scalebox{1.25}{$\mapsto$}$}; 

\draw[thick] (3,0) to[out=-90,in=-90] (4,0);  
\draw[thick] (2,0) to[out=-90,in=-90] (3.5,-0.5);  
\draw[thick] (1,0) to[out=-90,in=-90] (3,-1);  
\draw[thick] (0,0) to[out=-90,in=-90] (2,-1.3);  

\draw[thick] (1.5,0)--(2,-1);  
\draw[thick] (2.5,0)--(3,-.7); 
\draw[thick] (3.5,0)--(3.5,-.2);

\end{tikzpicture}
\]
 Now pairs of trees with matching numbers of leaves give rise to pairs of tangles that can be glued together. 
 For example, for the element $\frac{T_+}{T_-}$ we get the following link 
\[
\begin{tikzpicture}[x=.5cm, y=.5cm,
    every edge/.style={
        draw,
      postaction={decorate,
                    decoration={markings}
                   }
        }
]

\node (bbb) at (-3,-.5) {$\scalebox{1.25}{$\,$}$}; 
\draw[thick] (2,0) to[out=90,in=90] (3,0);  
\draw[thick] (2.5,0.5) to[out=90,in=90] (4,0);  
\draw[thick] (1,0) to[out=90,in=90] (3.5,.75);  
\draw[thick] (0,0) to[out=90,in=90] (1.5,.95);  

\draw[dashed] (-.5,0)--(4.5,0); 

\draw[thick] (2.5,0)--(2.5,.2); 
\draw[thick] (3.5,0)--(3.5,.5); 
\draw[thick] (1.5,0)--(1.5,.6);  

\node (bbb) at (-3.5,0) {$\scalebox{1.25}{$\CL(T_+,T_-)=$}$}; 

\draw[thick] (3,0) to[out=-90,in=-90] (4,0);  
\draw[thick] (2,0) to[out=-90,in=-90] (3.5,-0.5);  
\draw[thick] (1,0) to[out=-90,in=-90] (3,-1);  
\draw[thick] (0,0) to[out=-90,in=-90] (2,-1.3);  

\draw[thick] (1.5,0)--(2,-1);  
\draw[thick] (2.5,0)--(3,-.7); 
\draw[thick] (3.5,0)--(3.5,-.2);

\end{tikzpicture}
\]
See 
Figure~\ref{fig2} for more examples.

In the case of positive elements, 
the bottom tree
can always be chosen to have the following shape 
\[
\begin{tikzpicture}[x=.5cm, y=.5cm,
    every edge/.style={
        draw,
      postaction={decorate,
                    decoration={markings}
                   }
        }
]

\draw[thick] (0,0)--(4,-4)--(8,0);
\draw[thick] (4.5,-3.5)--(1,0);
\draw[thick] (5,-3)--(2,0);
\draw[thick] (7,-1)--(6,0);
\draw[thick] (7.5,-.5)--(7,0);
\draw[thick] (4,-4)--(4,-4.5);


\node (aaaa) at (5,-1) {$\scalebox{1}{$\ldots$}$}; 

\end{tikzpicture}
\]
 which corresponds to a standard tangle, to be found in Figures~\ref{fig2} and~\ref{fig5}.
 For this reason 
the links produced by elements of $F_+$ are simply denoted by $\CL(T_+)$. 
\begin{figure}[htb]
\begin{center}
\raisebox{-0mm}{\includegraphics[scale=0.7]{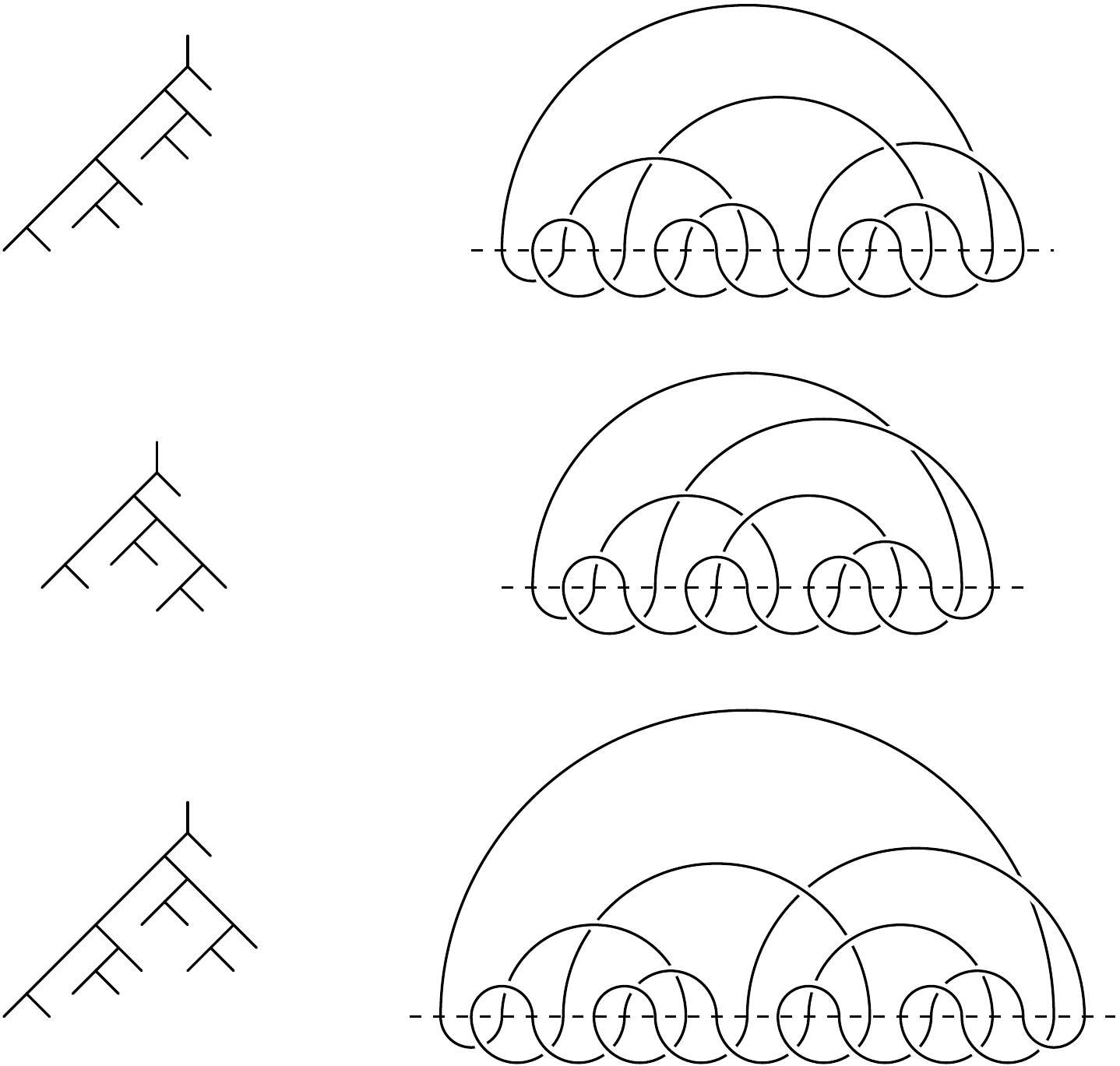}}
\caption{Positive Thompson knots $3_1,3_1^*,4_1$.}\label{fig2}
\end{center}
\end{figure}

In summary, the links associated with positive elements of the Thompson group are defined as special closures of tangles determined by a single rooted binary plane tree. As we can see from Figure~\ref{fig2}, the positive and negative trefoil knot $3_1$ and $3_1^*$, as well as the figure-eight knot $4_1$, arise from this construction. 

In the rest of this section, we explain why positive Thompson link diagrams are closures of arborescent tangles. The class of arborescent tangles is the minimal class of tangles closed under tangle composition, and containing all rational tangles. The latter are described by finite sequences of integers, viewed as the coefficients of a continued fraction expansion of a rational number~\cite{Co}. The closure of an arborescent tangle is described by a finite rooted plane tree with integer vertex weights. Each weight $w$ gives rise to a twist region with $|w|$ crossings. The orientation of these crossings, as well as the interconnections between these twist regions, are determined by the plane tree in the following way. The root vertex corresponds to a horizontal twist region, in which crossings are called positive if their strand going from the bottom left to the top right is above the other strand. 
If the weight is zero, then we have just two horizontal lines.
The vertices adjacent to the root vertex correspond to vertical twist regions attached to this horizontal twist region. The order in which they are attached is determined by the plane cyclic arrangement of the branches around the root vertex. We keep the convention that the overcrossing strand of a positive crossing is going from the bottom left to the top right. In the end, this means that arborescent tangles whose weights carry the same sign give rise to alternating links. The vertices at distance two from the root give again rise to horizontal twist regions, and so on. The three examples depicted in Figure~1 illustrate this construction. A more detailed definition can be found in~\cite{BS,Ga}.

Going back to positive Thompson links, let us start with the simplest positive element, represented by a binary top tree with just two leaves. The corresponding link diagram is a union of two overlapping circles, which we may interpret as the closure of an arborescent tangle in many ways. We choose quite an unusual interpretation, with two vertices of weight zero, corresponding to crossingless tangles, as shown in Figure~\ref{fig3}. The reason for this choice is the inductive argument, which will soon follow. More precisely, we will associate a planar, weighted tree to a binary tree so that there are two consecutive vertices of the planar tree associated to each trivalent vertex of the binary tree. The construction is inductive. 
 Indeed, the tree $T$ may be constructed from a single caret, by adding a sequence of carets   to it.
Correspondingly,  the weighted tree, may be obtained by adding to the weighted tree (depicted in Figure \ref{fig3}) a pair of vertices/edges for each caret.
In the next lemma, we describe how to add such vertices/edges.
\begin{figure}[htb]
\begin{center}
\raisebox{-0mm}{\includegraphics[scale=1.0]{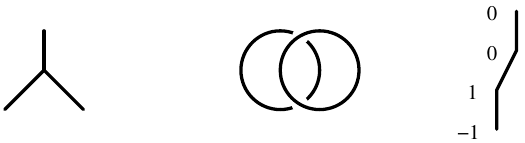}}
\caption{Basic arborescent tangle.}\label{fig3}
\end{center}
\end{figure}

\begin{lemma}\label{addcarets}
Let $T$ be a plane rooted binary tree and denote by $t$ the weighted tree corresponding to the link $\CL(T)$. 
By adding a caret below one of the leaves of $T$ we obtain a new tree $T'$ and correspondingly a new weighted $t'$.
The tree $t'$ may be obtained from $t$ in the following way depending on where we add the new caret 
\begin{itemize}
\item[a)] 
if we attach a   caret below  a left leaf of 
  $T$, two new vertices 
are added to $t$: one with weight $1$ and connected to a terminal vertex of weight $-1$, another one of weight $-1$   connected by one edge only to this new vertex of weight $1$,
\[
\begin{tikzpicture}[x=.35cm, y=.35cm,
    every edge/.style={
        draw,
      postaction={decorate,
                    decoration={markings}
                   }
        }
]
\node at (-1.75,0.5) {$\scalebox{1}{$\,$}$};


 \draw[thick] (1,1.5) -- (1,1);
\draw[thick] (0,0) -- (1,1)--(2,0);

\filldraw[fill=gray!40!white, draw=black] (1,0) -- (1.5,.5)--(2,0)--(1,0);

\end{tikzpicture}\;\;
\begin{tikzpicture}[x=.35cm, y=.35cm,
    every edge/.style={
        draw,
      postaction={decorate,
                    decoration={markings}
                   }
        }
]

\node at (-1.5,0.5) {$\scalebox{1}{$\mapsto$}$};

 \draw[thick] (2,2.5) -- (2,2);
\draw[thick] (0,0) -- (2,2)--(4,0);
 \draw[thick] (1,1) -- (2,0);

\filldraw[fill=gray!40!white, draw=black] (3,0) -- (3.5,.5)--(4,0)--(3,0);

\end{tikzpicture}
\]
\[
\begin{tikzpicture}[x=.35cm, y=.35cm,
    every edge/.style={
        draw,
      postaction={decorate,
                    decoration={markings}
                   }
        }
]

\node at (0,3.5) {$\scalebox{1}{$\vdots$}$};
\node at (-1.35,2) {$\scalebox{.75}{$-1/0$}$};
\node at (-1,1) {$\scalebox{.75}{$1$}$};
\node at (-1.35,0) {$\scalebox{.75}{$-1$}$};
 

 \draw[thick] (0,2) -- (0,0);
 \fill (0,0)  circle[radius=1.5pt];
\fill (0,1)  circle[radius=1.5pt];
\fill (0,2)  circle[radius=1.5pt];
 
\node at (2,2.5) {$\scalebox{1}{$\,$}$};

\end{tikzpicture}
\quad 
\begin{tikzpicture}[x=.35cm, y=.35cm,
    every edge/.style={
        draw,
      postaction={decorate,
                    decoration={markings}
                   }
        }
]

\node at (0,3.5) {$\scalebox{1}{$\vdots$}$};
\node at (-1.35,2) {$\scalebox{.75}{$-1/0$}$};
\node at (-1,1) {$\scalebox{.75}{$1$}$};
\node at (-1.35,0) {$\scalebox{.75}{$-1$}$};

\node at (-1,-1) {$\scalebox{.75}{$1$}$};
\node at (-1.35,-2) {$\scalebox{.75}{$-1$}$};


 \draw[thick] (0,2) -- (0,0);
\draw[thick, red] (0,0) -- (0,-2);
\fill (0,0)  circle[radius=1.5pt];
\fill (0,1)  circle[radius=1.5pt];
\fill (0,2)  circle[radius=1.5pt];

\fill[red] (0,-1)  circle[radius=1.5pt];
\fill[red] (0,-2)  circle[radius=1.5pt];

\node at (-3.5,0.5) {$\scalebox{1}{$\mapsto$}$};
\node at (2,2.5) {$\scalebox{1}{$,$}$};

\end{tikzpicture}
\]
where we put $-1/0$ to indicate that the vertices may have weight $-1$ or $0$ (the zero weights can occur only on the root and on the adjacent vertex).
\item[b)] 
if we attach a caret below a right-inner leaf
of $T$, two new vertices 
are added to $t$: one with weight $1$ and connected to a   vertex of weight $-1$, 
one with weight $-1$ connected by an edge
only to this new vertex of weight $1$, 
\[
\begin{tikzpicture}[x=.35cm, y=.35cm,
    every edge/.style={
        draw,
      postaction={decorate,
                    decoration={markings}
                   }
        }
]
\node at (-1.75,0.5) {$\scalebox{1}{$\,$}$};


 \draw[thick] (1,1.5) -- (1,1);
\draw[thick] (0,0) -- (1,1)--(2,0);

\filldraw[fill=gray!40!white, draw=black] (1,0) -- (.5,.5)--(0,0)--(1,0);

\end{tikzpicture}\;\;
\begin{tikzpicture}[x=.35cm, y=.35cm,
    every edge/.style={
        draw,
      postaction={decorate,
                    decoration={markings}
                   }
        }
]

\node at (-1.5,0.5) {$\scalebox{1}{$\mapsto$}$};

 \draw[thick] (2,2.5) -- (2,2);
\draw[thick] (0,0) -- (2,2)--(4,0);
 \draw[thick] (3,1) -- (2,0);

\filldraw[fill=gray!40!white, draw=black] (1,0) -- (.5,.5)--(0,0)--(1,0);

\end{tikzpicture}
\]
\[\qquad
\begin{tikzpicture}[x=.35cm, y=.35cm,
    every edge/.style={
        draw,
      postaction={decorate,
                    decoration={markings}
                   }
        }
]


\node at (0,3.5) {$\scalebox{1}{$\vdots$}$};
\node at (-1.35,2) {$\scalebox{.75}{$-1$}$};
\node at (-1,1) {$\scalebox{.75}{$1$}$};
\node at (-1.35,0) {$\scalebox{.75}{$-1$}$};


 \draw[thick] (0,2) -- (0,0);
 \fill (0,0)  circle[radius=1.5pt];
\fill (0,1)  circle[radius=1.5pt];
\fill (0,2)  circle[radius=1.5pt];

\end{tikzpicture}
\quad
\begin{tikzpicture}[x=.35cm, y=.35cm,
    every edge/.style={
        draw,
      postaction={decorate,
                    decoration={markings}
                   }
        }
]


\node at (0,3.5) {$\scalebox{1}{$\vdots$}$};
\node at (-1.35,2) {$\scalebox{.75}{$-1$}$};
\node at (-1,1) {$\scalebox{.75}{$1$}$};
\node at (-1.35,0) {$\scalebox{.75}{$-1$}$};

\node at (2,0) {$\scalebox{.75}{$-1$}$};
\node at (2.35,1) {$\scalebox{.75}{$1$}$};


 \draw[thick] (0,2) -- (0,0);
\draw[thick, red] (0,2) -- (1,1)--(1,0);
\fill (0,0)  circle[radius=1.5pt];
\fill (0,1)  circle[radius=1.5pt];
\fill (0,2)  circle[radius=1.5pt];

\fill[red] (1,1)  circle[radius=1.5pt];
\fill[red] (1,0)  circle[radius=1.5pt];

\node at (-3.5,0.5) {$\scalebox{1}{$\mapsto$}$};

\node at (3,2.5) {$\scalebox{1}{$,$}$};

\end{tikzpicture}
\]
%
%
%
%
%
%
%
%
%
%
%
\item[c)] 
if we attach a caret below the right-most leaf
of $T$, two new vertices 
are added to $t$: one with weight $1$ and connected to a   vertex of weight $0$, 
one with weight $-1$ connected by an edge
only to this new vertex of weight $1$, 
\[
\begin{tikzpicture}[x=.35cm, y=.35cm,
    every edge/.style={
        draw,
      postaction={decorate,
                    decoration={markings}
                   }
        }
]
\node at (-1.75,0.5) {$\scalebox{1}{$\,$}$};


 \draw[thick] (1,1.5) -- (1,1);
\draw[thick] (0,0) -- (1,1)--(2,0);
\filldraw[fill=gray!40!white, draw=black] (1,0) -- (.5,.5)--(0,0)--(1,0);

\node at (0,-1.25) {$\scalebox{1}{$\;$}$};

\node at (2,-.35) {$\scalebox{.75}{$\star$}$};

\end{tikzpicture}\;\;
\begin{tikzpicture}[x=.35cm, y=.35cm,
    every edge/.style={
        draw,
      postaction={decorate,
                    decoration={markings}
                   }
        }
]
\filldraw[fill=gray!40!white, draw=black] (1,0) -- (.5,.5)--(0,0)--(1,0);

\node at (-1.5,0.5) {$\scalebox{1}{$\mapsto$}$};

 \draw[thick] (2,2.5) -- (2,2);
\draw[thick] (0,0) -- (2,2)--(4,0);
 \draw[thick] (3,1) -- (2,0);

\node at (0,-1.25) {$\scalebox{1}{$\;$}$};

\node at (4,-.35) {$\scalebox{.75}{$\star$}$};

\end{tikzpicture}
\]
\[\qquad
\begin{tikzpicture}[x=.35cm, y=.35cm,
    every edge/.style={
        draw,
      postaction={decorate,
                    decoration={markings}
                   }
        }
]

\node at (-1,3) {$\scalebox{.75}{$0$}$};
\node at (-1,1) {$\scalebox{.75}{$1$}$};
\node at (-1.35,0) {$\scalebox{.75}{$-1$}$};
\node at (-1,2) {$\scalebox{.75}{$0$}$};



 \draw[thick] (0,3) -- (0,0);
 \fill (0,0)  circle[radius=1.5pt];
\fill (0,1)  circle[radius=1.5pt];
\fill (0,2)  circle[radius=1.5pt];
\fill (0,3)  circle[radius=1.5pt];

\node at (0,-.75) {$\scalebox{1}{$\vdots$}$};

\end{tikzpicture}
\quad
\begin{tikzpicture}[x=.35cm, y=.35cm,
    every edge/.style={
        draw,
      postaction={decorate,
                    decoration={markings}
                   }
        }
]

\node at (-1,3) {$\scalebox{.75}{$0$}$};
\node at (-1,1) {$\scalebox{.75}{$1$}$};
\node at (-1.35,0) {$\scalebox{.75}{$-1$}$};
\node at (-1,2) {$\scalebox{.75}{$0$}$};


\node at (2,0) {$\scalebox{.75}{$-1$}$};
\node at (2.35,1) {$\scalebox{.75}{$1$}$};


 \draw[thick] (0,3) -- (0,0);
\draw[thick, red] (0,2) -- (1,1)--(1,0);
\fill (0,0)  circle[radius=1.5pt];
\fill (0,1)  circle[radius=1.5pt];
\fill (0,2)  circle[radius=1.5pt];
\fill (0,3)  circle[radius=1.5pt];

\fill[red] (1,1)  circle[radius=1.5pt];
\fill[red] (1,0)  circle[radius=1.5pt];

\node at (0,-.75) {$\scalebox{1}{$\vdots$}$};

\node at (-3.5,0.5) {$\scalebox{1}{$\mapsto$}$};

\node at (3,2.5) {$\scalebox{1}{$,$}$};

\end{tikzpicture}
\]
where 
we  put the label $\star$ below the right-most leaf of the tree. 
\end{itemize}
Here the gray triangles represent 
an arbitrary subtree and the new edges are drawn in red.
\end{lemma}    
\begin{proof}
We now give a pictorial proof of the three rules.
For all the three cases here follow the portions of the knot diagrams which are affected by the addition of the caret 
(first the link diagram before adding the caret, then the one after the addition\footnote{For a) we only draw the case where the weights of the last three vertices are $-1$, $1$, $-1$, the case where the first vertex has zero weight can proved in a similar manner}).
The same diagrams represent the the arborescent link described by the weighted trees drawn above.
\begin{eqnarray*}
%
%
%
%
%
%
%
%
\begin{tikzpicture}[x=.5cm, y=.5cm,
    every edge/.style={
        draw,
      postaction={decorate,
                    decoration={markings}
                   }
        }
]

\node at (2,0.1) {$\scalebox{.65}{$\rotatebox{45}{$\scalebox{0.7}{$\ddots$}$}$}$};

 \draw[thick] (1,0) to[out=90,in=90] (2,0.6);  

\draw[thick] (1,0) to[out=-90,in=-90] (2.5,-0.5);  
 \draw[thick] (2,-0.2) to[out=-90,in=-90] (3,0);  
\draw[thick] (1.5,-.8)--(1.5,-1.3);
 \draw[thick] (1.5,.8)--(1.5,1.4);
\draw[thick] (1.5,.5)--(1.5,-.4);
\draw[thick] (2.5,.1)--(2.5,-.1);

 \draw[thick] (1.75,.4)to[out=45,in=135](2.25,.4);
\draw[thick] (2,.3)--(2,.4);

\node at (-1.5,-0.5) {$\scalebox{1}{$a)$}$};

\node at (0,-1.2) {$\;$};

\end{tikzpicture}
&
\begin{tikzpicture}[x=.5cm, y=.5cm,
    every edge/.style={
        draw,
      postaction={decorate,
                    decoration={markings}
                   }
        }
]

\node at (2,0.1) {$\scalebox{.65}{$\rotatebox{45}{$\scalebox{0.7}{$\ddots$}$}$}$};

\draw[thick] (0,0) to[out=90,in=90] (1,0);  
\draw[thick] (.5,0.5) to[out=90,in=90] (2,0.7);  

\draw[thick] (1,0) to[out=-90,in=-90] (2.5,-0.5);  
\draw[thick] (0,0) to[out=-90,in=-90] (1.5,-.75);  
\draw[thick] (2,-0.2) to[out=-90,in=-90] (3,0);   

\draw[thick] (.5,.1)--(.5,-.5);
\draw[thick] (.5,-1.3)--(.5,-.9);
\draw[thick] (1.5,.9)--(1.5,-.5);
\draw[thick] (1.5,1.2)--(1.5,1.5);
\draw[thick] (2.5,.1)--(2.5,-.1);

\node at (-1.5,-0.5) {$\scalebox{1}{$\mapsto\; $}$};

 \draw[thick] (1.75,.5)to[out=45,in=135](2.25,.5);
 \draw[thick] (2,.5)--(2,.4);

\node at (0,-1.2) {$\;$};

\end{tikzpicture}\\
\begin{tikzpicture}[x=.5cm, y=.5cm,
    every edge/.style={
        draw,
      postaction={decorate,
                    decoration={markings}
                   }
        }
]

 \draw[thick] (1,0) to[out=90,in=90] (2,0);  

\draw[thick] (1,0) to[out=-90,in=-90] (2.5,-0.5);  
 \draw[thick] (2,0) to[out=-90,in=-90] (3,0);  
\draw[thick] (1.5,-.8)--(1.5,-1.3);
 \draw[thick] (1.5,.5)--(1.5,1.2);
\draw[thick] (1.5,.1)--(1.5,-.4);
\draw[thick] (2.5,.1)--(2.5,-.1);

\node at (-1.5,-0.5) {$\scalebox{1}{$b)$}$};

\node at (0,-1.2) {$\;$};
  \fill [color=white,opacity=1] (1.25,0.2) -- (-.5,0.2)--(-0.5,-.2)--(1.25,-.2)--(1.25,0.2);
  
\node at (1,0) {$\scalebox{.65}{$\rotatebox{0}{$\scalebox{0.7}{$\ldots$}$}$}$};

\end{tikzpicture}
&
\begin{tikzpicture}[x=.5cm, y=.5cm,
    every edge/.style={
        draw,
      postaction={decorate,
                    decoration={markings}
                   }
        }
]

\draw[thick] (0,0) to[out=90,in=90] (1.5,0.5);  
\draw[thick] (1,0) to[out=90,in=90] (2,0);  

\draw[thick] (1,0) to[out=-90,in=-90] (2.5,-0.5);  
\draw[thick] (0,0) to[out=-90,in=-90] (1.5,-.75);  
\draw[thick] (2,0) to[out=-90,in=-90] (3,0);  
\draw[thick] (.6,-.6)--(.6,.5);
\draw[thick] (.6,-1)--(.6,-1.4);
\draw[thick] (.6,.9)--(.6,1.4);
\draw[thick] (1.5,.1)--(1.5,-.4);
\draw[thick] (2.5,.1)--(2.5,-.1);

\node at (-1.5,-0.5) {$\scalebox{1}{$\mapsto\; $}$};

\node at (0,-1.2) {$\;$};
  \fill [color=white,opacity=1] (.25,0.2) -- (-.5,0.2)--(-0.5,-.2)--(.25,-.2)--(.25,0.2);

\node at (0,0) {$\scalebox{.65}{$\rotatebox{0}{$\scalebox{0.7}{$\ldots$}$}$}$};

\end{tikzpicture}\\
%
%
%
%
%
%
%
%
\begin{tikzpicture}[x=.5cm, y=.5cm,
    every edge/.style={
        draw,
      postaction={decorate,
                    decoration={markings}
                   }
        }
]

 \draw[thick] (0,0) to[out=90,in=90] (2,0);  

\draw[thick] (1,-1) to[out=-90,in=-90] (2,0);   

\draw[thick] (1.1,.4)--(1.25,-1);
\draw[thick] (1,.7)--(1,1.1);
\draw[thick] (1.25,-1.2)--(1.25,-1.8);

\node at (-2.5,-0.5) {$\scalebox{1}{$c)$}$};

\node at (0,-1.2) {$\;$};
\node at (0,-.5) {$\scalebox{.65}{$\rotatebox{0}{$\scalebox{0.7}{$\ldots$}$}$}$};

\end{tikzpicture}
& 
\begin{tikzpicture}[x=.5cm, y=.5cm,
    every edge/.style={
        draw,
      postaction={decorate,
                    decoration={markings}
                   }
        }
]

\draw[thick] (0,0) to[out=90,in=90] (1.5,0.5);  
\draw[thick] (1,0) to[out=90,in=90] (2,0);  

\draw[thick] (1,0) to[out=-90,in=-90] (2,0);  
\draw[thick] (0,-1) to[out=-90,in=-90] (1.5,-.5);   

\draw[thick] (1.5,.2)--(1.5,-.2);
\draw[thick] (.5,.4)--(.5,-1);
\draw[thick] (.5,.7)--(.5,1.1);
\draw[thick] (.5,-1.4)--(.5,-1.9);

\node at (-1.5,-0.5) {$\scalebox{1}{$\mapsto\; $}$};

\node at (0,-1.2) {$\;$};
\node at (0,-.5) {$\scalebox{.65}{$\rotatebox{0}{$\scalebox{0.7}{$\ldots$}$}$}$};

\end{tikzpicture}
%
%
%
%
%
%
%
%
%
%
\end{eqnarray*}
\end{proof}

An easy application of the previous lemma yields the following result.
\begin{proposition}\label{inclusion}
For every $g\in F_+$, the link $\CL(g)$   admits an arborescent description with two adjacent vertices of weight zero, one of which is the root, all other weights $\pm 1$,
and with all the vertices of weight $1$ having degree $2$.
\end{proposition}
We observe that all crossings corresponding to vertices of weight $-1$ belong to the bottom tangle. 
A global example is shown in Figure~\ref{fig5}: the full binary tree with sixteen leaves, together with its link diagram and the arborescent tangle description obtained from the above procedure. From this example, we can guess an explicit arborescent tangle description for the Thompson links associated with full binary trees; they all represent the trivial link with two components.

\begin{figure}[htb]
\begin{center}
\raisebox{-0mm}{\includegraphics[scale=0.8]{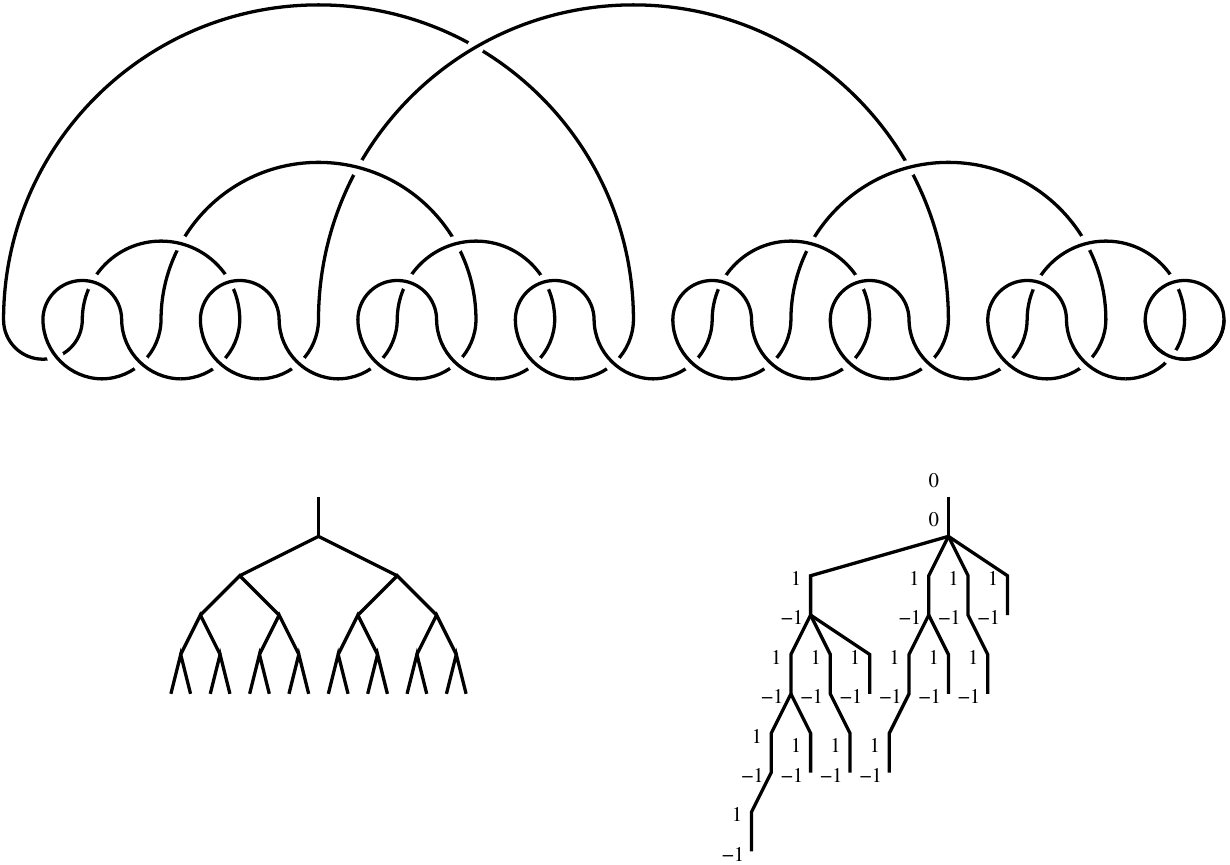}}
\caption{Full binary tree with sixteen leaves.}\label{fig5}
\end{center}
\end{figure}

Arborescent tangles are made up of finitely many twist regions that are wired in an arborescent pattern. In practice, these twist regions are chosen to be as large as possible. We will do the contrary, in order to 
realize the closures of all the arborescent tangles with only positive coefficients
(recall that these tangles are all alternating)  as positive Thompson links.
We call a finite rooted plane tree \textbf{bipartite}, if it carries a bipartite structure, encoded by vertex weights $\pm 1$, so that the root, as well as all leaves, carry the weight $-1$, and the vertices of weight $1$ have degree $2$. 

Our goal is to prove Theorem \ref{arborescent}. First we present the following lemma and then we will prove the first inclusion.
\begin{lemma}\label{lemma2}
For any $a\in\IN$, the following move does not affect the corresponding arborescent links 
\[
\begin{tikzpicture}[x=.35cm, y=.35cm,
    every edge/.style={
        draw,
      postaction={decorate,
                    decoration={markings}
                   }
        }
]

\node at (-1.5,0) {$\scalebox{1}{$\ldots$}$};
 \node at (0,-.75) {$\scalebox{.5}{$a$}$};
\node at (.75,-.75) {$\scalebox{.5}{$\pm 1$}$}; 
 \node at (1.75,-.75) {$\scalebox{.5}{$\mp1$}$}; 
\node at (2.75,-.75) {$\scalebox{.5}{$\pm 1$}$}; 
 

 \draw[thick] (-.5,0) -- (3,0);
 \fill (0,0)  circle[radius=1.5pt];
\fill (1,0)  circle[radius=1.5pt];
\fill (2,0)  circle[radius=1.5pt];
\fill (3,0)  circle[radius=1.5pt];

\end{tikzpicture}\begin{tikzpicture}[x=.35cm, y=.35cm,
    every edge/.style={
        draw,
      postaction={decorate,
                    decoration={markings}
                   }
        }
]

\node at (-3,0) {$\scalebox{1}{$\leftrightarrow$}$};

\node at (-1.5,0) {$\scalebox{1}{$\ldots$}$};
 \node at (0,-.75) {$\scalebox{.5}{$a$}$};
 

  \fill (0,0)  circle[radius=1.5pt];

 \draw[thick] (-.5,0) -- (0,0);

\end{tikzpicture}
\]
\end{lemma}
\begin{proof}
We only consider the case where the last 3 vertices have weights -1, 1, -1; the case 1, -1, 1
 is analogous.
The claim follows at once by drawing the corresponding tangles
\[
\begin{tikzpicture}[x=.5cm, y=.5cm,
    every edge/.style={
        draw,
      postaction={decorate,
                    decoration={markings}
                   }
        }
]


\draw[thick] (-3,4.5)--(-5,4.5)--(-5,2.5)--(-3,2.5)--(-3,4.5);
\draw[thick] (-5,4)--(-5.5,4);
\draw[thick] (-5,3)--(-5.5,3);


\draw[thick] (-3,3)--(-2,3);
\draw[thick] (-3,4)--(-2,4); 

\draw[thick] (-1,4)--(-2,4);
\draw[thick] (-1,3)--(-1,1);
 \draw[thick] (-2,3) to[out=-90,in=180] (-1,0);  
 \draw[thick] (0,3) to[out=-90,in=90] (2.25,1);

\draw[thick] (-1,4)--(0,3);
\draw[thick] (-1,3)--(-.7,3.3); 
\draw[thick] (-.3,3.7)--(0,4); 


\draw[thick] (0,1)--(1.25,1);
\draw[thick] (0,0) to[out=-45,in=-135] (1.25,0);  
 
\draw[thick] (-1,0)--(0,1);
\draw[thick] (-1,1)--(-.7,.7); 
\draw[thick] (-.3,.3)--(0,0); 

\draw[thick] (1.25,1)--(2.25,0);
\draw[thick] (1.25,0)--(1.55,.3); 
\draw[thick] (1.95,.7)--(2.25,1);

 \node at (-4,3.5) {$\scalebox{1}{$a$}$};

\end{tikzpicture}
\quad 
\begin{tikzpicture}[x=.5cm, y=.5cm,
    every edge/.style={
        draw,
      postaction={decorate,
                    decoration={markings}
                   }
        }
]

 \node at (-6.5,3.5) {$\scalebox{1}{$=$}$};


\draw[thick] (-3,4.5)--(-5,4.5)--(-5,2.5)--(-3,2.5)--(-3,4.5);
\draw[thick] (-5,4)--(-5.5,4);
\draw[thick] (-5,3)--(-5.5,3);

\draw[thick] (-3,3)--(-2.5,3);  
\draw[thick] (-3,4)--(-2.5,4);  

\node at (3,-0.5) {$\scalebox{1}{$\;$}$};
\node at (-4,3.5) {$\scalebox{1}{$a$}$};
 
\end{tikzpicture}
\]
\end{proof}

\begin{proposition}
The closure of every arborescent tangle associated with a finite rooted plane bipartite tree is realised as a link of a positive element of the Thompson group.
\end{proposition}
\begin{proof}
Bipartite trees have root with weight $-1$, while the roots of the weighted trees associated with arborescent links constructed from elements of $F_+$ have weight $0$.
Starting from a bipartite tree, first we manipulate the bipartite tree in such a way to obtain a weighted tree whose
 corresponding link is  (up to isotopy) the same, but it has a new root of weight $0$  connected only to another vertex of weight $0$, which in turn is connected to a vertex whose weight is $1$.
\[\begin{tikzpicture}[x=.35cm, y=.35cm,
    every edge/.style={
        draw,
      postaction={decorate,
                    decoration={markings}
                   }
        }
]


\node at (0.5,-.75) {$\scalebox{1}{$\vdots$}$};
\node at (-.5,2) {$\scalebox{.75}{$-1$}$};
\node at (-1,1) {$\scalebox{.75}{$1$}$};
\node at (-1.35,0) {$\scalebox{.75}{$-1$}$};

\node at (2,0) {$\scalebox{.75}{$-1$}$};
\node at (2.35,1) {$\scalebox{.75}{$1$}$};


 \draw[thick] (0.5,2) -- (-.5,1)-- (-.5,0);
\draw[thick] (0.5,2) -- (1.5,1)--(1.5,0);
\fill (-0.5,0)  circle[radius=1.5pt];
\fill (-0.5,1)  circle[radius=1.5pt];
\fill (0.5,2)  circle[radius=1.5pt];

\fill (1.5,1)  circle[radius=1.5pt];
\fill (1.5,0)  circle[radius=1.5pt];

\node at (0.5,.5) {$\scalebox{1}{$\ldots$}$};

\end{tikzpicture}
\quad
\begin{tikzpicture}[x=.35cm, y=.35cm,
    every edge/.style={
        draw,
      postaction={decorate,
                    decoration={markings}
                   }
        }
]


\node at (0.5,-.75) {$\scalebox{1}{$\vdots$}$};
\node at (0.5,.5) {$\scalebox{1}{$\ldots$}$};
\node at (-.5,2) {$\scalebox{.75}{$-1$}$};
\node at (-.2,3) {$\scalebox{.75}{$1$}$};
\node at (-.5,4) {$\scalebox{.75}{$-1$}$};
\node at (-.2,5) {$\scalebox{.75}{$1$}$};
\node at (-1,1) {$\scalebox{.75}{$1$}$};
\node at (-1.35,0) {$\scalebox{.75}{$-1$}$};

\node at (2,0) {$\scalebox{.75}{$-1$}$};
\node at (2.35,1) {$\scalebox{.75}{$1$}$};


 \draw[thick] (0.5,2) -- (-.5,1)-- (-.5,0);
\draw[thick] (0.5,2) -- (1.5,1)--(1.5,0);
\fill (-0.5,0)  circle[radius=1.5pt];
\fill (-0.5,1)  circle[radius=1.5pt];
\fill (0.5,2)  circle[radius=1.5pt];

\fill (1.5,1)  circle[radius=1.5pt];
\fill (1.5,0)  circle[radius=1.5pt];

\fill (0.5,2)  circle[radius=1.5pt];
\fill (0.5,3)  circle[radius=1.5pt];
\fill (0.5,4)  circle[radius=1.5pt];
\fill (0.5,5)  circle[radius=1.5pt];
 
\node at (-3.5,0.5) {$\scalebox{1}{$\mapsto\; $}$};
 \draw[thick] (.5,2)--(.5,5);

\end{tikzpicture}
\quad
\begin{tikzpicture}[x=.35cm, y=.35cm,
    every edge/.style={
        draw,
      postaction={decorate,
                    decoration={markings}
                   }
        }
]


\node at (0.5,.5) {$\scalebox{1}{$\ldots$}$};

\node at (0.5,-.75) {$\scalebox{1}{$\vdots$}$};
\node at (-.5,2) {$\scalebox{.75}{$-1$}$};
\node at (-.2,3) {$\scalebox{.75}{$1$}$};
\node at (-.5,4) {$\scalebox{.75}{$-1$}$};
\node at (-.2,5) {$\scalebox{.75}{$1$}$};
\node at (-.2,6) {$\scalebox{.75}{$0$}$};
\node at (-.2,7) {$\scalebox{.75}{$0$}$};

\node at (-1,1) {$\scalebox{.75}{$1$}$};
\node at (-1.35,0) {$\scalebox{.75}{$-1$}$};

\node at (2,0) {$\scalebox{.75}{$-1$}$};
\node at (2.35,1) {$\scalebox{.75}{$1$}$};


 \draw[thick] (.5,2)--(.5,7);

 \draw[thick] (0.5,2) -- (-.5,1)-- (-.5,0);
\draw[thick] (0.5,2) -- (1.5,1)--(1.5,0);
\fill (-0.5,0)  circle[radius=1.5pt];
\fill (-0.5,1)  circle[radius=1.5pt];
\fill (0.5,2)  circle[radius=1.5pt];

\fill (1.5,1)  circle[radius=1.5pt];
\fill (1.5,0)  circle[radius=1.5pt];

\fill (0.5,2)  circle[radius=1.5pt];
\fill (0.5,3)  circle[radius=1.5pt];
\fill (0.5,4)  circle[radius=1.5pt];
\fill (0.5,5)  circle[radius=1.5pt];
\fill (0.5,6)  circle[radius=1.5pt];
\fill (0.5,7)  circle[radius=1.5pt];

\node at (-5.5,0.5) {$\scalebox{1}{$\mapsto\; $}$};
\node at (-3.5,0.5) {$\scalebox{1}{$t= $}$};

\end{tikzpicture}
\]
The first step does not affect the corresponding link thanks to Lemma \ref{lemma2}.
The second step clearly does not change the link.
 Now  repeated applications of Lemma \ref{addcarets} produce an element $g$ of $F_+$ whose corresponding link $\CL(g)$ is the same as the arborescent link associated with $t$. 
Indeed, starting from the tree with two leaves (which is depicted in Figure \ref{fig3}) we may add carets to the tree with $2$ leaves in such a way that
the corresponding weighted tree is $t$.
\end{proof}
We now prove that the converse inclusion holds.
\begin{proposition}
Every positive Thompson link is the closure of an arborescent tangle associated with a finite rooted plane bipartite tree.
\end{proposition}
\begin{proof}
Given an element $g\in F_+$, 
let $(t_+,t_-)$ be the minimal pair of trees representing $g$.
It is easy to check that applying the transformation depicted below does not affect the corresponding positive Thompson link
\[
\begin{tikzpicture}[x=.35cm, y=.35cm,
    every edge/.style={
        draw,
      postaction={decorate,
                    decoration={markings}
                   }
        }
]
\node at (-1.5,0) {$\scalebox{1}{$\frac{t_+}{t_-}=$}$};


 \draw[thick] (1,1.5) -- (1,1);
\draw[thick] (0,0) -- (1,1)--(2,0);
\filldraw[fill=gray!40!white, draw=black] (2,0) -- (1,1)--(0,0)--(2,0);

 \draw[thick] (1,-1.5) -- (1,-1);
\draw[thick] (0,0) -- (1,-1)--(2,0);

\node at (0,-2.5) {$\scalebox{1}{$\;$}$};

 \node at (1,-.25) {$\scalebox{.5}{$\ldots$}$};

\end{tikzpicture}\;\;
\begin{tikzpicture}[x=.35cm, y=.35cm,
    every edge/.style={
        draw,
      postaction={decorate,
                    decoration={markings}
                   }
        }
]
\filldraw[fill=gray!40!white, draw=black] (2,0) -- (1,1)--(0,0)--(2,0);

\node at (-1.5,0) {$\scalebox{1}{$\mapsto$}$};

 \draw[thick] (3,0) -- (1.5,1.5);
 \draw[thick] (4,0) -- (2,2);
\draw[thick] (2.5,3) -- (2.5,2.5);

\draw[thick] (0,0) -- (2.5,2.5)--(5,0);

\draw[thick] (0,0) -- (2.5,-2.5)--(5,0);
\draw[thick] (2.5,-3) -- (2.5,-2.5);
 \draw[thick] (4,0) -- (4.5,-.5);
 \draw[thick] (3,0) -- (4,-1);
 \draw[thick] (2,0) -- (3.5,-1.5);

\node at (0,-1.25) {$\scalebox{1}{$\;$}$};

 \node at (1.25,-.25) {$\scalebox{.5}{$\ldots$}$};

\end{tikzpicture}
\]
where the gray triangle represents the rest of the tree.
Thanks to this move, we may assume that in the tree constructed by means of Lemma \ref{addcarets} 
 the two vertices of weight 0 have degree 1 and 2. Below the corresponding tree is displayed, along with a bipartite tree yielding the same link (see Lemma \ref{lemma2}).
\[
\begin{tikzpicture}[x=.35cm, y=.35cm,
    every edge/.style={
        draw,
      postaction={decorate,
                    decoration={markings}
                   }
        }
]

\node at (0.5,-.75) {$\scalebox{1}{$\vdots$}$};
\node at (-.5,2) {$\scalebox{.75}{$-1$}$};
\node at (-.2,3) {$\scalebox{.75}{$1$}$};
\node at (-.5,4) {$\scalebox{.75}{$-1$}$};
\node at (-.2,5) {$\scalebox{.75}{$1$}$};
\node at (-.2,6) {$\scalebox{.75}{$0$}$};
\node at (-.2,7) {$\scalebox{.75}{$0$}$};

\node at (-.2,1) {$\scalebox{.75}{$1$}$};
\node at (-.5,0) {$\scalebox{.75}{$-1$}$};
 

 \draw[thick] (.5,0)--(.5,7);

 \fill (0.5,0)  circle[radius=1.5pt];
\fill (0.5,1)  circle[radius=1.5pt];
\fill (0.5,2)  circle[radius=1.5pt];

\fill (0.5,2)  circle[radius=1.5pt];
\fill (0.5,3)  circle[radius=1.5pt];
\fill (0.5,4)  circle[radius=1.5pt];
\fill (0.5,5)  circle[radius=1.5pt];
\fill (0.5,6)  circle[radius=1.5pt];
\fill (0.5,7)  circle[radius=1.5pt];

\end{tikzpicture}
\quad
\begin{tikzpicture}[x=.35cm, y=.35cm,
    every edge/.style={
        draw,
      postaction={decorate,
                    decoration={markings}
                   }
        }
]

\node at (0.5,-.75) {$\scalebox{1}{$\vdots$}$};
\node at (-.5,2) {$\scalebox{.75}{$-1$}$};
\node at (-.2,3) {$\scalebox{.75}{$1$}$};
\node at (-.5,4) {$\scalebox{.75}{$-1$}$};
\node at (-.2,5) {$\scalebox{.75}{$1$}$}; 

\node at (-.2,1) {$\scalebox{.75}{$1$}$};
\node at (-.5,0) {$\scalebox{.75}{$-1$}$};
 

 \draw[thick] (.5,0)--(.5,5);

 \fill (0.5,0)  circle[radius=1.5pt];
\fill (0.5,1)  circle[radius=1.5pt];
\fill (0.5,2)  circle[radius=1.5pt];

\fill (0.5,2)  circle[radius=1.5pt];
\fill (0.5,3)  circle[radius=1.5pt];
\fill (0.5,4)  circle[radius=1.5pt];
\fill (0.5,5)  circle[radius=1.5pt];

\node at (-2.5,1) {$\scalebox{1}{$\mapsto$}$};

\end{tikzpicture}
\quad
\begin{tikzpicture}[x=.35cm, y=.35cm,
    every edge/.style={
        draw,
      postaction={decorate,
                    decoration={markings}
                   }
        }
]

\node at (0.5,-.75) {$\scalebox{1}{$\vdots$}$};
\node at (-.5,2) {$\scalebox{.75}{$-1$}$};

\node at (-.2,1) {$\scalebox{.75}{$1$}$};
\node at (-.5,0) {$\scalebox{.75}{$-1$}$};
 

 \draw[thick] (.5,0)--(.5,2);

 \fill (0.5,0)  circle[radius=1.5pt];
\fill (0.5,1)  circle[radius=1.5pt];
\fill (0.5,2)  circle[radius=1.5pt];

\fill (0.5,2)  circle[radius=1.5pt];

\node at (-2.5,1) {$\scalebox{1}{$\mapsto$}$};

\end{tikzpicture}
\]
\end{proof}

\section{Trees with positive signs}
The aim of this section is to prove Corollary \ref{positive}.
Our strategy is to show that the links associated with plane rooted trees with all strictly positive weights, that is all the vertices have weight positive and non-zero, can also be realised by bipartite trees. 
We give a proof by induction on the number of vertices of the tree by presenting an explicit algorithm.

We start with the simplest tree, a single vertex with a positive weight~$k$. The corresponding tangle, a single twist region with~$k$ crossings, is realised by the rooted plane bipartite tree depicted on the left of Figure~\ref{fig7}. The three trees in that figure represent the same tangle. Here
the first equality is a consequence of Lemma \ref{lemma2}, the second one is obvious.
\begin{figure}[htb]
\begin{center}
\raisebox{-0mm}{\includegraphics[scale=1.0]{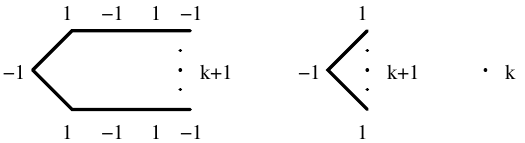}}
\caption{Realising a single twist region with~$k$ crossings.}\label{fig7}
\end{center}
\end{figure}


Suppose now that the number of vertices is strictly larger than $1$.
%
%
%
%
%
%
%
%
%
%
%
%
%
%
%
%
%
The \textbf{first step} is to subdivide each edge into four parts by inserting 3 vertices, two of weight -1 and one of weight 1. In the following lemma we show that this does not affect the corresponding link.
\begin{lemma}\label{lemma3}
For any $a, b\in\IN$, the following move does not affect the corresponding arborescent links 
\[
\begin{tikzpicture}[x=.35cm, y=.35cm,
    every edge/.style={
        draw,
      postaction={decorate,
                    decoration={markings}
                   }
        }
]

\node at (-1,0) {$\scalebox{1}{$\ldots$}$};
\node at (5,0) {$\scalebox{1}{$\ldots$}$};
\node at (0,-.75) {$\scalebox{.5}{$a$}$};
\node at (.75,-.75) {$\scalebox{.5}{$\pm 1$}$}; 
 \node at (1.75,-.75) {$\scalebox{.5}{$\mp1$}$}; 
\node at (3,-.75) {$\scalebox{.5}{$\pm 1$}$}; 
 \node at (4,-.75) {$\scalebox{.5}{$b$}$}; 


 \draw[thick] (0,0) -- (4,0);
 \fill (0,0)  circle[radius=1.5pt];
\fill (1,0)  circle[radius=1.5pt];
\fill (2,0)  circle[radius=1.5pt];
\fill (3,0)  circle[radius=1.5pt];
\fill (4,0)  circle[radius=1.5pt];

\end{tikzpicture}\begin{tikzpicture}[x=.35cm, y=.35cm,
    every edge/.style={
        draw,
      postaction={decorate,
                    decoration={markings}
                   }
        }
]

\node at (-3,0) {$\scalebox{1}{$\leftrightarrow$}$};

\node at (-1,0) {$\scalebox{1}{$\ldots$}$};
\node at (2,0) {$\scalebox{1}{$\ldots$}$};
\node at (0,-.75) {$\scalebox{.5}{$a$}$};
\node at (1,-.75) {$\scalebox{.5}{$b$}$}; 
 

 \draw[thick] (0,0) -- (1,0);
 \fill (0,0)  circle[radius=1.5pt];
\fill (1,0)  circle[radius=1.5pt];

\end{tikzpicture}
\]
%
%
%
%
%
%
%
%
%
%
%
%
%
%
%
%
%
\end{lemma}
\begin{proof}
We only consider the case where we add the the sequence -1, 1, -1. 
The claim follows at once by drawing the corresponding tangles
\[
\begin{tikzpicture}[x=.5cm, y=.5cm,
    every edge/.style={
        draw,
      postaction={decorate,
                    decoration={markings}
                   }
        }
]


\draw[thick] (-3,4.5)--(-5,4.5)--(-5,2.5)--(-3,2.5)--(-3,4.5);
\draw[thick] (-5,4)--(-5.5,4);
\draw[thick] (-5,3)--(-5.5,3);


\draw[thick] (-3,3)--(-2,3);
\draw[thick] (-3,4)--(-2,4); 

\draw[thick] (-1,4)--(-2,4);
\draw[thick] (-1,3)--(-1,1);
 \draw[thick] (-2,3) to[out=-90,in=180] (-1,0);  
 \draw[thick] (0,3) to[out=-90,in=90] (2.25,1);

\draw[thick] (-1,4)--(0,3);
\draw[thick] (-1,3)--(-.7,3.3); 
\draw[thick] (-.3,3.7)--(0,4); 

\draw[thick] (1,0)--(2.5,0)--(2.5,-1.5)--(1,-1.5)--(1,0);
\draw[thick] (-5,4)--(-5.5,4);
\draw[thick] (-5,3)--(-5.5,3);

\draw[thick] (0,1)--(1.25,1);
\draw[thick] (0,0)--(0,-2);
\draw[thick] (0,-2) to[out=-90,in=-90] (1.25,-1.5);  
\draw[thick] (2.25,-1.5)--(2.25,-2);

\draw[thick] (-1,0)--(0,1);
\draw[thick] (-1,1)--(-.7,.7); 
\draw[thick] (-.3,.3)--(0,0); 

\draw[thick] (1.25,1)--(2.25,0);
\draw[thick] (1.25,0)--(1.55,.3); 
\draw[thick] (1.95,.7)--(2.25,1);

\node at (1.75,-0.75) {$\scalebox{1}{$b$}$};
\node at (-4,3.5) {$\scalebox{1}{$a$}$};

\end{tikzpicture}
\quad 
%
%
%
%
%
%
%
%
%
\begin{tikzpicture}[x=.5cm, y=.5cm,
    every edge/.style={
        draw,
      postaction={decorate,
                    decoration={markings}
                   }
        }
]

 \node at (-6.5,3.5) {$\scalebox{1}{$=$}$};


\draw[thick] (-3,4.5)--(-5,4.5)--(-5,2.5)--(-3,2.5)--(-3,4.5);
\draw[thick] (-5,4)--(-5.5,4);
\draw[thick] (-5,3)--(-5.5,3);

\draw[thick] (-2,4.5)--(0,4.5)--(0,2.5)--(-2,2.5)--(-2,4.5);
\draw[thick] (0,4)--(.5,4);
\draw[thick] (0,3)--(.5,3);  

\draw[thick] (-3,3)--(-2,3);  
\draw[thick] (-3,4)--(-2,4);  

\node at (3,-0.5) {$\scalebox{1}{$\;$}$};
\node at (-4,3.5) {$\scalebox{1}{$a$}$};
\node at (-1,3.5) {$\scalebox{1}{$\rotatebox{90}{$\scalebox{1}{$b$}$}$}$};

\end{tikzpicture}
\]
\end{proof}
In the \textbf{second step} we need to replace the vertices whose weight is strictly greater than 1
by suitable bipartite trees.
\begin{lemma}
Any 
vertex of weight $a\in\IN$ can be replaced by 
\[
\begin{tikzpicture}[x=.35cm, y=.35cm,
    every edge/.style={
        draw,
      postaction={decorate,
                    decoration={markings}
                   }
        }
]

\node at (-1,0) {$\scalebox{1}{$\ldots$}$};
\node at (3,0) {$\scalebox{1}{$\ldots$}$};
\node at (0,.5) {$\scalebox{.5}{$1$}$};
\node at (1,.5) {$\scalebox{.5}{$a$}$}; 
 \node at (2,.5) {$\scalebox{.5}{$1$}$};  

 \draw[thick] (0,0) -- (2,0);
 \fill (0,0)  circle[radius=1.5pt];
\fill[red] (1,0)  circle[radius=1.5pt];
\fill (2,0)  circle[radius=1.5pt];

 \node at (2,-2.75) {$\scalebox{.5}{$\;$}$};

\end{tikzpicture}\begin{tikzpicture}[x=.35cm, y=.35cm,
    every edge/.style={
        draw,
      postaction={decorate,
                    decoration={markings}
                   }
        }
]

\node at (-3,0) {$\scalebox{1}{$\leftrightarrow$}$};

\node at (-1,0) {$\scalebox{1}{$\ldots$}$};
\node at (3,0) {$\scalebox{1}{$\ldots$}$};
\node at (0,.5) {$\scalebox{.5}{$1$}$};
\node at (1,.5) {$\scalebox{.5}{$-1$}$}; 
 \node at (2,.5) {$\scalebox{.5}{$1$}$};  

 \draw[thick] (0,0) -- (2,0);
 \draw[thick] (0,-1) -- (1,0);
 \draw[thick] (.5,-1) -- (1,0);
 \draw[thick] (2,-1) -- (1,0);
 \fill (0,0)  circle[radius=1.5pt];
\fill[red] (1,0)  circle[radius=1.5pt];
\fill (2,0)  circle[radius=1.5pt];
 
\fill (0,-1)  circle[radius=1.5pt];
\fill (.5,-1)  circle[radius=1.5pt];
\fill (2,-1)  circle[radius=1.5pt];
 \node at (0,-1.5) {$\scalebox{.5}{$1$}$};  
 \node at (.5,-1.5) {$\scalebox{.5}{$1$}$};  
 \node at (2,-1.5) {$\scalebox{.5}{$1$}$};  
 \node at (1,-2.75) {$\scalebox{.5}{$a+1$}$};

\draw [decorate,decoration={brace,amplitude=2pt,mirror},xshift=1pt,yshift=0pt] (-.2,-2) -- (2,-2) node [black,midway,xshift=0.2cm] {\footnotesize $\;$};
\node at (1.25,-0.8) {$\scalebox{.6}{$\ldots$}$};

\end{tikzpicture}
\]
without affecting the corresponding link.
\end{lemma}
\begin{proof}
In the figure below we represent both tangles
\[
\begin{tikzpicture}[x=.5cm, y=.5cm,
    every edge/.style={
        draw,
      postaction={decorate,
                    decoration={markings}
                   }
        }
]


\draw[thick] (0,0)--(1,1);
\draw[thick] (0,1)--(.3,.7); 
\draw[thick] (.7,.3)--(1,0);

\draw[thick] (1,0)--(2,1);
\draw[thick] (1,1)--(1.3,.7); 
\draw[thick] (1.7,.3)--(2,0); 

\node at (3,0.5) {$\scalebox{1}{$\ldots$}$};

\draw[thick] (4,0)--(5,1);
\draw[thick] (4,1)--(4.3,.7); 
\draw[thick] (4.7,.3)--(5,0); 

 \node at (2.5,-1) {$\scalebox{1}{$a$}$};  

\node at (6,0.5) {$\scalebox{1}{$\ldots$}$};
\node at (-1,0.5) {$\scalebox{1}{$\ldots$}$};

\draw [decorate,decoration={brace,amplitude=2pt,mirror},xshift=1pt,yshift=0pt] (-.2,-.2) -- (5,-.2) node [black,midway,xshift=0.2cm] {\footnotesize $\;$};
 
\end{tikzpicture}
\quad
\begin{tikzpicture}[x=.5cm, y=.5cm,
    every edge/.style={
        draw,
      postaction={decorate,
                    decoration={markings}
                   }
        }
]

 \node at (-3.5,0.5) {$\scalebox{1}{$=\; $}$};


\draw[thick] (-1,0)--(0,1);
\draw[thick] (-1,1)--(-.7,.7); 
\draw[thick] (-.3,.3)--(0,0); 

\draw[thick] (0,0)--(1,1);
\draw[thick] (0,1)--(.3,.7); 
\draw[thick] (.7,.3)--(1,0);

\draw[thick] (3,0)--(4,1);
\draw[thick] (3,1)--(3.3,.7); 
\draw[thick] (3.7,.3)--(4,0); 

\node at (2,0.5) {$\scalebox{1}{$\ldots$}$};
\node at (6,0.5) {$\scalebox{1}{$\ldots$}$};
\node at (-2,0.5) {$\scalebox{1}{$\ldots$}$};

\draw[thick] (4,1)--(5,0);
\draw[thick] (4,0)--(4.3,.3); 
\draw[thick] (4.7,.7)--(5,1); 

 \node at (2.5,-1) {$\scalebox{1}{$a+1$}$};

\draw [decorate,decoration={brace,amplitude=2pt,mirror},xshift=1pt,yshift=0pt] (-1.2,-.2) -- (4,-.2) node [black,midway,xshift=0.2cm] {\footnotesize $\;$};
 
\end{tikzpicture}
\]
\end{proof}

Finally if there are two consecutive vertices of weight 1 (this happens whenever some of the vertices of the original tree have weight 1) insert  3 vertices of weight -1, 1, -1 (this does not affect the corresponding link thanks to Lemma \ref{lemma3}).
The same sequence can be attached to all leaves with sign +1, leaving us with a bipartite tree.
This concludes the proof of Corollary \ref{positive}.

\section*{Acknowledgements}
We thank the referee for his or her particularly attentive perusal of the manuscript, which resulted in many improvements in the presentation of the results of this paper.
The authors acknowledge the support by the Swiss National Science foundation through the SNF project no. 178756 (Fibred links, L-space covers and algorithmic knot theory).

\end{document}